\documentclass[12pt, leqno]{amsart}

\usepackage{amsfonts}
\usepackage{amssymb}
\usepackage{amsmath}
\usepackage{amscd}
\usepackage{float}
\usepackage{graphicx}
\usepackage{amsfonts}

\usepackage{color}
\usepackage{cancel}
\usepackage[cp1250]{inputenc}
\usepackage[T1]{fontenc}

\usepackage{multirow}
\usepackage{multicol}
\usepackage{hyperref}

\usepackage{color}
\usepackage{xypic}

\usepackage[cmtip,arrow]{xy}

\usepackage[normalem]{ulem}
\newtheorem{proposition}{Proposition}
\newtheorem{definition}{Definition}
\newtheorem{remark}{Remark}

\def\J{\mathcal{J}}
\def\P{\mathcal{P}}

\begin{document}

\title{Fuss-Catalan  Triangles}

\author{F. Aicardi}
 \address{Sistiana 56, Trieste IT}
 \email{francescaicardi22@gmail.com}

\begin{abstract}  For  each $p>0$  we  define by  recurrence  a  triangle   $T^p(n,k)$ whose  rows    sum to the Fuss-Catalan numbers $ \frac{1}{p n+1}\binom{pn+1}{n}$, generalizing the known Catalan triangle corresponding to the case $p=2$. (In fact, $T^p(n,k)$  has  an explicit formula  counting simple lattice paths). Moreover, for  some small values of  $p$,  the  signed  sums  turn out to be  known sequences.
\end{abstract}

\keywords{}

\subjclass{05A10,05A19,57M50}

\date{}
\maketitle

\section{Introduction}
The  Catalan numbers
$$   C_n= \frac{1}{n+1}\binom{2n}{n}$$
can be  obtained  by  adding the  rows of the  Catalan triangle $T(n,k)$  defined  by the following   recurrence and initial  conditions:
\begin{equation}\label{tr1}
T(0,0)=1, \quad
  T(n,0)=1,\quad  T(n,n)=0, \quad  \quad n>0 ;
  \end{equation}
  \begin{equation} \label{tr2}  T(n,k)= T(n,k-1)+T(n-1, k), \quad   \quad  n>1, \quad 0<k < n .
 \end{equation}

 Observe that  recurrence (\ref{tr2}) is  equivalent to  recurrence:
  \begin{equation} \label{tr3}  T(n,k)= \sum_{j=0}^k T(n-1, j), \quad   \quad  n>1, \quad 0<k < n .
 \end{equation}

Now, let's denote  $A^p_n$ the  Fuss-Catalan numbers
 $$A^p_n:= \frac{1}{p n+1}\binom{pn+1}{n},$$
 generalizing the  Catalan numbers, being $C_n=A^2_n$.

 In \cite{Aic},  we have defined  by  recurrence  a  triangle  whose  n-th row row  sums to  $A^4_n$, see  also \cite{Oeis}.  Such  a triangle was not seen as  a generalization of the triangle  $T(n,k)$,  as  in fact  it was.

 In the  next  section  we generalize  both  recurrences \ref{tr2}  and  \ref{tr3}  defining,  for each integer $p>0$,  a triangle whose   rows  add to  $A^p_n$. The triangle  $T(n,k)$ corresponds to case $p=2$.

\section{  Fuss-Catalan triangles   }

 For  each $p>0$, we define  a  triangle   $T^p(n,k)$.

\begin{definition}\rm  For  every  $p>0$,  the  integers $T^p(n,k)$  are  given for  $n\ge 0$,  $0\le k \le n$  by the initial conditions
\begin{equation}\label{E0}
T^p(0,0)=1; \quad
 T^p(n,n)=0 \quad \text{ for} \quad  n>0
  \end{equation}
  and  the  recurrence, for $n>0$  and  $0\le  k<n$:
 \begin{equation} \label{E2}   T^p(n,k)=  \sum_{j=0}^{k}  \binom{k-j+p-2}{p-2} T^p(n-1, j).
 \end{equation}
\end{definition}

\begin{proposition}
\label{the2}\rm  The  integers $T^p(n,k)$  can be  defined for  $n\ge 0$, $1-p \le k \le n$  by the initial conditions \ref{E0} together  with
\begin{equation}\label{E1}
 \quad T^p(n,-j)=0  \quad \text{ for} \quad  j=1,\dots, p-1 ,
  \end{equation}
  and  the  recurrence for $n>0$  and  $0\le  k<n$:
 \begin{equation} \label{E3}   T^p(n,k)= T^p(n-1, k)+ \sum_{j=1}^{p-1} (-1)^{j-1} \binom{p-1}{j} T^p(n,k-j).
 \end{equation}
 \end{proposition}

\begin{proof} The  equivalence  of  equations  (\ref{E3})  and  (\ref{E2}) is deduced  from the  following  identity  satisfied  by  the  binomial  coefficients:
 \begin{equation}\label{E4}
 \binom{m+h}{h}= \sum_{r=1}^{h+1}(-1)^{r-1}\binom{h+1}{r} \binom{m+h-r}{h},
  \end{equation}
 by setting  $h=p-2$  and  $m=k-j$.
\end{proof}

\begin{remark}\rm Observe that,  for  $p=2$,  recurrence (\ref{E3}) coincides  with recurrence (\ref{tr2}), and
recurrence (\ref{E2}) coincides  with recurrence (\ref{tr3}).
\end{remark}

The  following  proposition  has  been proved for $p=4$  in  \cite{Aic}.  For the other values of  $p>2$  see  Remark \ref{formul}.
\begin{proposition}  \label{P3}
The  integers $T^p(n,k)$    satisfy,  for  every  $n>0$
\begin{equation}\label{f3}  \sum_{k=0}^{n} T^p(n,k)=   A_n^p.  \end{equation}
  \end{proposition}

\begin{remark} \rm Note that  $A^1_n=1$  for  every $n$, and    $T^1(n,0)=1$ for $n>0$, while $T^1(n,k)=0$  for every $k>0$.
\end{remark}
Examples.
For  $p=5$  the  recurrence (\ref{E2}) and (\ref{E3}) are  respectively
$$ T^5(n,k)= T^5(n-1, k)+  \binom{4}{3}  T^5(n-1,k-1) + \binom{5}{3}  T^5(n-1,k-2) + \dots + \binom{k+3}{3} T^5(n-1,0) $$
and
$$ T^5(n,k)= T^5(n-1, k)-  4 T^5(n,k-1) + 6  T^5(n,k-2) -4 T^5(n,k-3)+ T^5(n,k-4).$$

Here  the  triangle $T^5(n,k)$  for  $n\le 10$,  $0\le k\le n$ (zeroes at left are omitted)
{\tiny
 $$\begin{array}{l|lllllllllll  |l}
          n   \backslash k &    & 0   & 1  &  2       &  3     & 4      &     5  &  6      &   7     & 8    &9  & sum \\ \hline
          0&    & 1  &   &       &     &      &       &        &        &        &        &  1  \\
          1&    & 1  & 0  &       &     &      &       &        &        &        &        & 1\\
          2&    & 1  &  4 &  0    &     &      &       &        &        &        &        & 5\\
          3&    &1   &  8 &  26  &  0   &      &       &        &        &        &        & 35\\
          4&    & 1  &  12&  68  &  204&  0    &       &        &        &        &        &  285\\
          5&    & 1  &  16&  126 &  616&  1771&  0     &        &        &        &        &  2530\\
          6&    &1   &  20 &  200 &  1300&  5850&    16380& 0      &       &        &        &  23751\\
          7&    &1  &   24 &  290 &    2320&   13485&    57536  &158224  &   0     &        &        & 231880\\
          8&    &1  &  28&  396   &  3740  &  26180  &  141372   &  581196&  1577532&   0    &        & 2330445\\
          9&    & 1 &  32&  518   &  5624   &  45695  &  292448  &  1498796&  5995184&  16112057&  0   & 23950355\\
          10&   & 1&  36&  656    &  8036    &  74046 &  543004  &  3258024&  16057404&  62891499& 167710664&  250543370\\
 \end{array}$$
    }
\begin{remark} \label{formul}\rm   An explicit formula for $T^p(n,k)$   (see  \cite{Kratt1}) is
     \begin{equation} T^p(n,k) =  \frac{(p-1)(n-k)}{(p-1)n+k } \binom {(p-1)n+k}{k},  \end{equation}
     giving   the number of  simple plane  lattice paths from     $(0,0)$  to  $((p-1)n-1,k)$ staying below the line $x=(p-1)y$, see \cite[Theorem 10.4.5]{Kratt2}. 
 \end{remark}

\section{Some  applications}

\subsection{Fuss-Catalan triangles  and   double planar partitions.  }
Here   we  show  how  the Fuss-Catalan triangles  $T^3$  and  $T^4$ enters two  problems  of  counting double planar partitions.

A set partition of $[n]$  can be  represented   by a  diagram of  arcs:   $n$ points on a  segment represent $[n]$  and  an arc connecting points $i$  and  $j$  means that  $i$ and  $j$  belong to  the  same block of the partition.  The  uniqueness of  the standard arc diagram is obtained  by avoiding the  arc connecting $i$ and $j$ ($i<j$) if in the  same block of $i$  and $j$ there is  $k$  such  that  $i<k<j$. So, a partition of  $n$  has  $0\le m\le   n-1$  arcs,  and  $n-m$ blocks.

A  partition  of $[n]$ is  said planar  if  the  arcs  do  not cross  each other. Let $\P_n$  be   the  set  planar partitions of  $[n]$.

A  double partition  is  a  pair  of partitions $(P_1,P_2)$  such  that  $P_1$ is  a  refinement of  $P_2$, i.e., each block of  $P_1$  is  contained in a  block of  $P_2$.

In a  double  planar partition  both  $P_1$  and  $P_2$  are planar.

A  double partition  (see  \cite{AAJ})  can  be  represented  by  a unique  diagram of  arcs  and  ties, the ties  being  drown  as dotted  arcs.   The  arcs represent the partition $P_1$.  A set of points  connected  by  arcs  is  a  block of  $P_1$   and    a  tie  between   two  blocks means  that  these  blocks  are contained in  the same block  of  $P_2$.

In a  diagram  of  arcs and  ties we introduce  the  definition of  box. Consider the half-plane  where  the  arcs and the ties lie.  The complement of the  set of points,  arcs and  ties to this  half-plane,   consists of  open regions;  only one of these regions,  say $U$,   is  unbounded.  A box  is either a point or  a connected  set of points, arcs, ties and  bounded regions  that borders only with $U$ in this  half plane, see Figure \ref{F1}.   A  diagram of  arcs and  ties  of  a  double planar partition of $[n]$  has  $1\le k\le n$  boxes.

\begin{figure}[H]
 \centering
 \includegraphics[scale=0.5]  {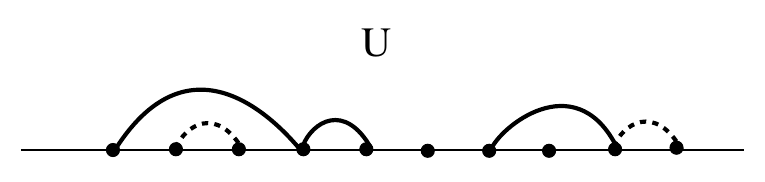}
 \caption{A diagram with arcs and ties representing a double partition of $[10]$ with  3  boxes}\label{F1}
\end{figure}

In \cite{Aic}  we  considered the  set $\J_n \subset \P_n$   having all  blocks of  two elements:  their  diagrams consist of  $n$  separate  arcs.  The cardinality of  $\J_n$ is  $C_n$  (see \cite{Jon}) and  we proved that  the  triangle $T(n,k)$ counts the  number of  such partitions  whose  diagrams have  $n-k$  boxes.

Moreover,  we  considered    the set of double  planar partitions $(P_1,P_2)$,  where  $P_1 \in J_n$  and  $P_2$ is planar. We showed that  its  cardinality   is  $A_n^4$,  and  that  the  triangle  $T^4(n,k)$  counts  the number  of  such double  partitions whose diagrams  of arcs and  ties  have  $n-k$  boxes.

It is  known that  the  Catalan number  $C_n$ counts also the cardinality of  $\P_n$.  Now,  let  $\P\P_n$ be the set   of  double planar partitions $(P_1,P_2)$,  with  $P_1,P_2\in \P_n$.

\begin{proposition}   The triangle $T^3(n,k)$  counts  the number  of planar double  partitions whose diagrams  of arcs and  ties  have  $n-k$  boxes. \end{proposition}

 \begin{proof} This follows from an  argument  analogous  to  that  used  in \cite{Aic}. Firstly, observe that a  diagram  with $n$ points and  zero  boxes does not exist,  while there is  exactly one   diagram - the  void  one - with  $0$ points  and  $0$ boxes,  so that, calling  $F^n_k$ the number of diagrams  with $n$ points   and  $k$  boxes,
  \begin{equation}\label{incon} F^0_0=1 ,  \quad    F^n_0=1. \end{equation}
   Consider  now the diagram of a double planar partition with $n$ points   and  $k>0$  boxes.  It  can be  obtained  by adding a point  and possibly arcs  and  ties  to  diagrams with  $n-1$ points  and  a  number  $j$  of boxes  from  $k-1$  to   $n-1$.  The  number of possibilities  depends on the  difference  $j-k$, and precisely  is $j-k+2$, see  Figure \ref{F2}.

\begin{figure}  [H]
 \centering
 \includegraphics[scale=0.5]{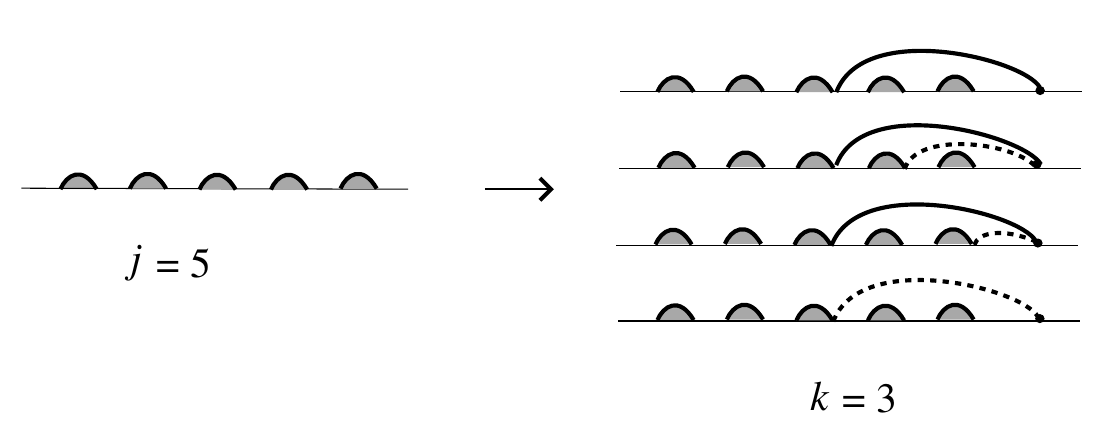}
 \caption{From a diagram in $\P\P_{n-1}$ with  $j=5$  boxes    we get  $j-k+2=4$ diagrams in $\P\P_n$ with  $k=3$ boxes }\label{F2}
\end{figure}
Therefore,   we get the  recurrence
$$ F^n_k=  \sum_{j=k-1}^{n-1} (j-k+2) F^{n-1}_j.$$
  Setting  $ F'(n,k):= F^n_{n-k}$,    we  get  from (\ref{incon}) $$ F'(0,0)=1 ,  \quad    F'(n,n)=0 $$
and   $$ F'(n,k)=  \sum_{j=0}^{k} (k-j+1) F'(n-1,j)  $$
  that    coincide   with eqs.  (\ref{E0}),(\ref{E2})  when  $p=3$.
\end{proof}
The cardinality of  $\P\P_n$  is  $A^3_n$  according to Proposition \ref{P3}.

\subsection{Signed sums}
It is known that  the  the  sums  of  the  elements of  the lines of  the  Catalan triangle  with  alternate signs give the  sequence  OEIS A000957.

 For  $p=3$,  the signed sum of  the elements of  the lines   $T^3(n,k)$  coincides  with  the  sequence  OEIS  A121545.

Here  the first 20 elements of  the signed  sums
$$ \sum_{k=0}^n (-1)^{n+k+1} T^p(n,k)$$  for  $1\le p\le 10$, $n \ge 0$:

{\small
 $ p=1: \ 1, -1, 1, -1, 1, -1, 1, -1, 1, -1, 1, -1, 1, -1, 1, -1, 1, -1, 1, -1...;$

$p=2: \ 1, 0, 1, 2, 6, 18, 57, 186, 622, 2120, 7338, 25724, 91144,325878, 1174281, 4260282, \\ 15548694, 57048048, 210295326,778483932... $

$p=3: \ 1, 1, 4, 17, 81, 412, 2192, 12049, 67891, 390041, 2276176, 13455356, 80402284, 484865032,\\
      2947107384, 18036248337,111046920567, 687345582787, 4274642610932, 26697307240777...
$

$ p= 4: \ 1, 2, 10, 60, 401, 2864, 21400, 165220, 1307620, 10552302,86499898, 718258092, 6028798270,\\
 51069425360, 436027711460,3748369529060, 32417585857772, 281855861501960, \\
 2462233783684360, 21601019155153160...$

$ p=5: \ 1, 3, 19, 147, 1266, 11649, 112120, 1114899, 11363779, 118094918, 1246589202, 13329000273,\\
 144060795984,1571321660184, 17274285045024, 191207254734611,2129186692124169,\\ 23835659616939801,
  268097631543422205,3028297833564071142...$

$ p=6: \ 1, 4, 31, 294, 3101, 34930, 411454, 5006462, 62435002,793804504, 10250825256, 134081332728, \\
1772711804059,23652022623548, 318057452070956, 4306309389830270,58654896651059898,\\ 803161037803720456,
 11049671703743489677,152662559948684096794...$

$ p= 7: \ 1, 5, 46, 517, 6456, 86054, 1199192, 17259269, 254564065, 3827611400, 58451445218, \\
904085802598, 14134199260676,222988679854904, 3545629711151376, 56762260319982341, \\ 914155220284080247,
 14800452986865247637, 240754641356101075792, 3932851435052726675272...$

 $ p= 8: \ 1, 6, 64, 832, 12006, 184848, 2974644, 49431816, 841737889,14610802796, 257565095316,\\
  4598660646472, 82987181003624, 1511233062645312, 27735977030373144, 512513462764716936,\\
  9526999968453976488, 178031888229562475262, 3342570896269418935944, 63022232949507822412392...$

 $ p=9: \ 1, 7, 85, 1255, 20551, 358915, 6550279, 123430375, 2383118590,46899735857, 937329679716,\\
  18972826484515, 388147257987230, 8012971961953665, 166715436675313360, \\3492227597228944871,
73589192430662773815, 1558880084609150482930,\\ 33177952185404603834690, 709111018877519274538705...$

$ p=10: \ 1, 8, 109, 1802, 33016, 644930, 13162171, 277323266,5986509067, 131716300848,\\ 2942973888251,
 66594588347396,1523025237247008, 35148006960677962, \\ 817473524614463614,19141969237892705538,
  450900300072682301930, \\ 10677226925045765840624, 254022383016749006834713,6068915034389074878940072...$
}

  \end{document}